\documentclass[11pt]{amsart}
\usepackage[T1]{fontenc}
\usepackage[utf8]{inputenc}
\usepackage{lmodern}
\usepackage{microtype}
\usepackage{xcolor}
\usepackage{amsmath,amssymb,amsthm,mathtools}
\usepackage{enumitem}
\usepackage{needspace}
\usepackage[a4paper,margin=32mm]{geometry}
\usepackage[hidelinks]{hyperref}
\hypersetup{
pdftitle={Irrational Seshadri constants from dihedral orbits},
pdfauthor={Antonio Laface and Luca Ugaglia}
}
\usepackage{tikz}

\newtheorem{maintheorem}{Theorem}
\newtheorem{theorem}{Theorem}[section]
\newtheorem{proposition}[theorem]{Proposition}
\newtheorem{lemma}[theorem]{Lemma}
\newtheorem{corollary}[theorem]{Corollary}
\theoremstyle{remark}

\newcommand{\PP}{\mathbb P}
\newcommand{\CC}{\mathbb C}
\newcommand{\RR}{\mathbb R}
\newcommand{\ZZ}{\mathbb Z}
\newcommand{\OO}{\mathcal O}
\newcommand{\eps}{\varepsilon}
\newcommand{\Bl}{\operatorname{Bl}}
\newcommand{\Pic}{\operatorname{Pic}}
\newcommand{\Sym}{\operatorname{Sym}}

\numberwithin{equation}{section}
\setlist[enumerate,1]{label=\textup{(\roman*)},leftmargin=*,itemsep=3pt}
\title[Irrational Seshadri constants]{Irrational Seshadri constants from dihedral orbits}
\author[A.~Laface]{Antonio Laface}
\address{
Departamento de Matem\'atica,
Universidad de Concepci\'on,
Casilla 160-C,
Concepci\'on, Chile}
\email{alaface@udec.cl}

\author[L.~Ugaglia]{Luca Ugaglia}
\address{
Dipartimento di Matematica e Informatica,
Universit\`a degli studi di Palermo,
Via Archirafi 34,
90123 Palermo, Italy}
\email{luca.ugaglia@unipa.it}
\date{}
\date{}
\begin{document}
\begin{abstract}
We construct an irrational one-point Seshadri constant on the
blow-up $X_9$ of $\PP^2$ at nine very general points. The divisor
$L=9H-3(E_1+\cdots+E_5)-2(E_6+\cdots+E_9)$ is ample and satisfies
$\eps(L;x)=2\sqrt5$ at a very general point $x\in X_9$.
To prove this, we establish that the Seshadri constant of
$\OO_{\PP^1\times\PP^1}(1,1)$ at ten very general points is
$1/\sqrt5$, completing the reflection approach proposed by Dionne
and Roth for the ten-point case.
We also prove that the same equality holds at a very general
free orbit of a fixed dihedral group of order ten. This yields
an irrational one-point Seshadri constant on the singular quotient
surface. A plane model of its minimal resolution and a deformation
of the blow-up centers then give the result on $X_9$.
\end{abstract}
\maketitle
\section*{Introduction}
The rationality of one-point Seshadri constants is a longstanding
problem in the local positivity of line bundles
\cite[Remark~5.1.13]{Lazarsfeld}. For an ample line bundle on a
smooth projective surface, a Seshadri constant strictly below the
volume bound is computed by a curve and is therefore rational.
An irrational example must consequently attain the volume bound,
with a nonsquare self-intersection. This makes the question closely
related to the problem of determining nef cones of blown-up surfaces.

Conditional results link this problem to conjectures on plane linear
systems. Dumnicki, K\"uronya, Maclean, and Szemberg show that the
Segre--Harbourne--Gimigliano--Hirschowitz conjecture implies the
existence of irrational one-point Seshadri constants on suitable
blow-ups of the plane \cite{DumnickiKuronyaMacleanSzemberg}.
Hanumanthu and Harbourne obtain conditional examples under the
weaker hypothesis that every irreducible curve of negative
self-intersection on the further blow-up at the evaluation point
is a $(-1)$-curve \cite[Theorem~2.4]{HanumanthuHarbourne}.
Nine points are the first possible case for plane blow-ups:
one-point Seshadri constants on blow-ups at at most eight points
are rational \cite[Remark~2.3]{HanumanthuHarbourne}.
Analogous conditional constructions on blow-ups of ruled surfaces
have since been obtained by Hanumanthu, Jacob, Suhas, and Singh,
assuming a conjectural description of the negative curves
\cite{HanumanthuJacobSuhasSingh}.

A different approach comes from Dionne and Roth's study of
multipoint Seshadri constants on $\PP^1\times\PP^1$.
Their reflection method \cite[Theorem~8.1]{DionneRoth} uses
Petrakiev's multiplicity estimate for a divisor along a curve
\cite[Lemma~2.1 and Corollary~2.2]{Petrakiev}.
In \cite[Section~9.2]{DionneRoth}, they explain how to obtain
the value $1/\sqrt5$ for ten points, provided a certain limiting
class remains nef after specializing the points to a curve of
bidegree $(3,1)$. Establishing nefness on this specialization
is the step left open in their argument.

We complete this approach for ten points and obtain an irrational
one-point Seshadri constant on the blow-up of the plane at nine
very general points. The key additional feature is that the ten
points on $\PP^1\times\PP^1$ can be constrained to form an orbit
of a fixed dihedral group. Explicit de Jonqui\`eres transformations
give a nef isotropic class on a special orbit; a normal-bundle
calculation then allows the reflection argument to be carried out
within the family of orbits of the same group. The resulting
multipoint equality descends to a one-point equality on the
singular quotient. We identify a plane model of its minimal
resolution, deform the blow-up centers, and apply two quadratic
transformations to obtain an explicit ample divisor on the blow-up
of nine very general plane points.

We now state the results precisely. All varieties are over $\CC$;
``very general'' means outside a countable union of proper closed
subsets of the relevant irreducible parameter space. For an ample
line bundle $L$ on a normal projective surface and distinct smooth
points $p_1,\ldots,p_r$, write
\[
\eps(L;p_1,\ldots,p_r)
=\sup\bigl\{c\geq0:\pi^*L-c(E_1+\cdots+E_r)
\text{ is nef}\bigr\},
\]
where $\pi$ blows up the points and $E_i$ is the exceptional curve
over $p_i$. The volume bound is $\eps\leq\sqrt{L^2/r}$.
Set $\mathbb F_0=\PP^1\times\PP^1$. Fix a primitive fifth root
$\zeta$ and a real number $u>1$, and let $G$ be generated by
\begin{equation}\label{eq:action}
r(x,z)=(\zeta x,\zeta^3z),\qquad
s(x,z)=(u/x,u^3/z).
\end{equation}
Here $x,z$ are affine coordinates on the two factors. These
automorphisms extend to $\mathbb F_0$ and generate a dihedral
group of order ten.
\begin{maintheorem}\label{thm:general-points}
The Seshadri constant of $\OO_{\mathbb F_0}(1,1)$ at a very
general free $G$-orbit is $1/\sqrt5$. In particular, the same
equality holds at ten very general points of $\mathbb F_0$.
\end{maintheorem}
The orbit assertion makes it possible to descend the polarization
$\OO_{\mathbb F_0}(10,10)$ to the quotient $Y=\mathbb F_0/G$.
The resulting ample Cartier divisor has square $20$ and one-point
Seshadri constant $2\sqrt5$ at a very general smooth point of $Y$;
this is Corollary~\ref{thm:quotient}. Resolving $Y$ and deforming
the centers of a plane model gives our main result.
\begin{maintheorem}\label{thm:plane-nine-points}
Let $p_1,\ldots,p_9$ be very general points of $\PP^2$, and let
$X_9=\Bl_{p_1,\ldots,p_9}\PP^2$. Denote by $H$ the pullback of the
class of a line and by $E_i$ the exceptional curve over $p_i$.
Then the divisor $L=9H-3(E_1+\cdots+E_5)-2(E_6+\cdots+E_9)$
is ample and satisfies
\[
\eps(L;x)=2\sqrt5
\]
at a very general point $x\in X_9$. 
\end{maintheorem}
Thus irrational one-point Seshadri constants occur for ample line
bundles on smooth rational surfaces of Picard number ten.
Section~\ref{sec:proof-theorem-one} proves
Theorem~\ref{thm:general-points}, including the construction of the
nef class and the reflection within the orbit family.
Section~\ref{sec:proof-theorem-two} proves
Theorem~\ref{thm:plane-nine-points} through the quotient, its
plane model, and the passage from special blow-up centers to
very general plane points.
\subsection*{Acknowledgements}
The authors were partially supported by Proyecto FONDECYT Regular
No.~1230287.
The second author is member of
INdAM - GNSAGA.
\section{Proof of Theorem~\ref{thm:general-points}}\label{sec:proof-theorem-one}
We construct a nef class on a special $G$-orbit and then apply a
reflection in a family of free orbits of the fixed group $G$.

\subsection{A nef class on a special orbit}\label{sec:nef}
Write $H_1,H_2$ for the pulled-back ruling classes on a blow-up
of $\mathbb F_0$, and $E$ for the sum of its ten exceptional
curves. For $w\in\{2,3\}$, let
$r_w(x,z)=(\zeta x,\zeta^wz)$, and for $c,d\in\CC^*$ put
\[
Z_w(c,d)
=\langle r_w\rangle\cdot(1,1)
 \,\cup\,\langle r_w\rangle\cdot(c,d).
\]
Assume $c^5\ne1$ and $d\notin\{c^w,c^{w-5}\}$.
The pullback action of $r_w$ on
$H^0(\mathbb F_0,\OO_{\mathbb F_0}(5,1))$
has exactly two three-dimensional weight spaces, of weights
$0$ and $w$. In affine coordinates, these are
$\langle 1,x^5,x^{5-w}z\rangle$ and
$\langle x^w,z,x^5z\rangle$, respectively; the remaining
weight spaces have dimension two.

In each of the two three-dimensional spaces, vanishing at
$(1,1)$ and $(c,d)$ imposes two independent linear conditions,
since $c^5\ne1$. These conditions determine a section up to
scalar, and the eigenvector property ensures that it vanishes
on all of $Z_w(c,d)$. We choose the following generators,
understood after bihomogenization:
\begin{align*}
\sigma_0&=c^5-dc^{5-w}+(dc^{5-w}-1)x^5
          +(1-c^5)x^{5-w}z,\\
\sigma_1&=d(c^5-1)x^w+(c^w-dc^5)z
          +(d-c^w)x^5z.
\end{align*}
\begin{lemma}\label{lem:birational}
The birational map
\[
\psi:\mathbb F_0\dashrightarrow\mathbb F_0,
\qquad
(x,z)\longmapsto
\left(x,\frac{\sigma_1(x,z)}{d\,\sigma_0(x,z)}\right)
\]
lifts to an isomorphism $\widetilde\psi:X\to X'$, where
$X=\Bl_{Z_w(c,d)}\mathbb F_0$,
$X'=\Bl_{Z_w(c,d^*)}\mathbb F_0$, and
$d^*=c^{2w-5}/d$.
Let $H_1,H_2$ and $H_1',H_2'$ be the pulled-back ruling
classes on $X$ and $X'$, respectively.
Label the exceptional classes $E_i$ and $E_i'$ so that
corresponding centers have the same first coordinate,
and put $E=\sum_{i=1}^{10}E_i$. Then
\[
\widetilde\psi^*H_1'=H_1,\qquad
\widetilde\psi^*H_2'=5H_1+H_2-E,\qquad
\widetilde\psi^*E_i'=H_1-E_i.
\]
\end{lemma}

\begin{proof}
The sections $\sigma_0,\sigma_1$ constructed above vanish
on $Z_w(c,d)$.
Use homogeneous coordinates $[x_0:x_1]$ and $[z_0:z_1]$
on the two factors, with $x=x_1/x_0$ and $z=z_1/z_0$.
Write their bihomogenizations as
$\sigma_i=A_i(x_0,x_1)z_0+B_i(x_0,x_1)z_1$.
The coefficient determinant is
\[
A_0B_1-A_1B_0
=-(c^w-d)(c^{5-w}d-1)
(x_1^5-x_0^5)(x_1^5-c^5x_0^5).
\]
Under our assumptions, this has ten simple zeros,
at $x=\zeta^\ell$ and $x=c\zeta^\ell$ for
$0\le\ell<5$.
At each zero the coefficient matrix has rank one,
since $B_0=(1-c^5)x_0^wx_1^{5-w}$ is nonzero;
elsewhere it is invertible.
Thus the two sections have no common component.
Since $(5H_1+H_2)^2=10$, their base scheme consists
precisely of the ten reduced points of $Z_w(c,d)$.
Blowing up these points resolves the pencil and gives
a birational morphism $\beta:X\to\mathbb F_0$ preserving
the first projection.
Away from the ten fibers containing the base points,
the coefficient matrix is invertible, so the induced
map on each fiber is an automorphism.
On each of the ten remaining fibers, the matrix has
rank one, so its strict transform is contracted.
Each exceptional curve $E_i$ maps isomorphically onto
the corresponding target ruling fiber, since the base
points are simple.
Hence the exceptional locus of $\beta$ consists exactly
of the ten disjoint $(-1)$-curves of classes $H_1-E_i$.
To identify their images, restrict the two sections to
$x=1$ and $x=c$. Away from the base points, one obtains
\[
\frac{\sigma_1(1,z)}{\sigma_0(1,z)}=d,
\qquad
\frac{\sigma_1(c,z)}{\sigma_0(c,z)}=c^{2w-5}.
\]
The factor $1/d$ in the definition of $\psi$
therefore normalizes the first image to $(1,1)$,
while the second is $(c,d^*)$.
The map is equivariant under
$r_w(x,z)=(\zeta x,\zeta^wz)$, since $\sigma_0$ and
$\sigma_1$ have weights $0$ and $w$.
The other contracted fibers consequently map to the
remaining points of $Z_w(c,d^*)$.
Contracting the ten disjoint $(-1)$-curves gives a smooth
surface through which $\beta$ factors.
The induced birational morphism from this surface to
$\mathbb F_0$ contracts no curves and is therefore an
isomorphism.
Thus $\beta$ is the blow-up of the ten points
$Z_w(c,d^*)$, giving the required isomorphism
$\widetilde\psi:X\to X'$.
Finally, the first target ruling pulls back to $H_1$,
and the second pulls back to the pencil class
$5H_1+H_2-E$.
The exceptional curve $E_i'$ pulls back to the strict
transform of the first-ruling fiber through the
corresponding source point, whose class is $H_1-E_i$.
This proves the pullback formulas.
\end{proof}
Let $\tau(x,z)=(z,x)$ be the factor exchange.
Exchanging the two target factors gives inverse base configuration
$Z_{w^{-1}}(c^{2w-5}/d,c)$, so the parameters change by
\begin{equation}\label{eq:step}
(c,d,w)\longmapsto(c^{2w-5}/d,c,w^{-1}),
\end{equation}
where $w^{-1}$ is taken modulo five. 
\begin{lemma}\label{lem:admissible-iteration}
Fix $u>1$, and define $(c_j,d_j,w_j)$ recursively by
\eqref{eq:step}, starting from $(c_0,d_0,w_0)=(u^3,u,2)$.
Then the weights alternate between $2$ and $3$, and
$c_j^5\ne1$ and $d_j\notin\{c_j^{w_j},c_j^{w_j-5}\}$
for every $j\ge0$.
\end{lemma}

\begin{proof}
Since $2$ and $3$ are inverse modulo five, the weights alternate
as asserted. Write $(c_j,d_j)=(u^{a_j},u^{b_j})$.
The recurrence \eqref{eq:step} sends $(a,b)$ to
$(-a-b,a)$ at even steps and to $(a-b,a)$ at odd steps.
Starting from $(a_0,b_0)=(3,1)$, we claim that
$|a_j|>|b_j|>0$, with equal signs at even indices and
opposite signs at odd indices. Indeed, in either case the next
pair has the required signs and satisfies
$|a_{j+1}|=|a_j|+|b_j|$ and $|b_{j+1}|=|a_j|$.
This proves the claim by induction.
Since $u>1$ and $a_j\ne0$, we have $c_j^5\ne1$.
Moreover, $|b_j|<|a_j|$, whereas both $|w_j|$ and
$|w_j-5|$ are at least two. Thus
$b_j\notin\{w_ja_j,(w_j-5)a_j\}$, which gives
$d_j\notin\{c_j^{w_j},c_j^{w_j-5}\}$.
\end{proof}

We start with the ten-point configuration
\[
P_0:=\tau(Z_2(u^3,u))=Z_3(u,u^3)=G\cdot(1,1).
\]
These points lie on the smooth rational curve
$\Gamma=\{z=x^3\}$ of bidegree $(3,1)$.
Using the parameter recurrence \eqref{eq:step}, we define
the subsequent configurations inductively by
\[
P_{j+1}
:=\tau\bigl(Z_{w_{j+1}}(c_{j+1},d_{j+1})\bigr)
=\tau\bigl(Z_{w_j^{-1}}(c_{j+1},c_j)\bigr)
=Z_{w_j}(c_j,c_{j+1}).
\]
For every $j\ge0$, set
\[
X_j:=\Bl_{P_j}\mathbb F_0.
\]

By Lemma~\ref{lem:admissible-iteration}, the construction of
Lemma~\ref{lem:birational} applies to every triple
$(c_j,d_j,w_j)$. Denote the resulting birational map by
$\psi_j:
\mathbb F_0\dashrightarrow\mathbb F_0$, and write
$\sigma_{0,j},\sigma_{1,j}$ for its defining sections.
Its base points are $Z_{w_j}(c_j,d_j)=\tau(P_j)$,
and the base points of the inverse are
$Z_{w_j}(c_j,c_j^{2w_j-5}/d_j)
=Z_{w_j}(c_j,c_{j+1})=P_{j+1}$.
Consider the birational map
\[
F_j=\psi_j\circ\tau:\mathbb F_0\dashrightarrow\mathbb F_0,
\qquad
F_j(x,z)=
\left(z,\frac{\sigma_{1,j}(z,x)}
{d_j\,\sigma_{0,j}(z,x)}\right).
\]
The map $F_j$ induces a birational morphism
$X_j\to\mathbb F_0$ contracting precisely the ten disjoint
$(-1)$-curves given by the strict transforms of the fibers
of the second projection through $P_j$.
Their images are the ten points of $P_{j+1}$, so this
morphism is the blow-up of $P_{j+1}$ and hence lifts to an
isomorphism $f_j:X_j\to X_{j+1}$.
Each step therefore first exchanges the source factors and
then applies $\psi_j$. The pencil defining the second
coordinate of $F_j$ has bidegree $(1,5)$.
Using the ruling classes and the sum of the exceptional
classes to mark each $X_j$, every pullback $f_j^*$ has
matrix
\begin{equation}\label{eq:T}
T=\begin{pmatrix*}[r]
0&1&0\\
1&5&10\\
0&-1&-1
\end{pmatrix*}
\end{equation}
on the symmetric marked space $\langle H_1,H_2,E\rangle$.
This is Dionne and Roth's operator
$T_{10}$, see~\cite{DionneRoth}.

\begin{proposition}\label{prop:nef}
On $X_0$, the class
$R=\alpha H_1+\beta H_2-E$ is nef, where
$\alpha=(5-\sqrt5)/2$ and $\beta=(5+\sqrt5)/2$.
\end{proposition}
\begin{proof}
The classes $\xi_j=T^j(H_2)$ are pullbacks of rulings under the
composites of the $f_j$, hence nef. Put $\lambda=(3+\sqrt5)/2$,
$\overline R=\beta H_1+\alpha H_2-E$, and $K=-2H_1-2H_2+E$.
Direct calculation gives $T(R)=\lambda R$,
$T(\overline R)=\lambda^{-1}\overline R$, $T(K)=K$, and
$H_2=(\beta/5)R+(\alpha/5)\overline R+K$. Thus
$\alpha\lambda^{-j}\xi_j\to R$. The nef cone is closed.
\end{proof}
\subsection{Reflection within the orbit family}\label{sec:reflection}
In a family of blow-ups of ${\mathbb F_0}$ at labeled disjoint sections, the
ruling and exceptional divisors identify the Picard groups of all
geometric fibers by their coefficients. We call this the marking.
A marked real class that is nef on one fiber, or on the geometric
generic fiber, is nef on very general fibers. Indeed, the relative
Hilbert schemes have countably many components proper over the
base, and intersections with the marking divisors are constant on
each component. Components parametrizing cycles of negative
intersection have proper closed images in either case.
We use the following form of Petrakiev's multiplicity argument;
compare~\cite[Lemma~2.1 and Corollary~2.2]{Petrakiev} and
\cite[Theorem~8.1]{DionneRoth}.
\begin{lemma}\label{lem:reflection}
Let $(B,0)$ be a smooth connected pointed complex curve, and let
$\mathcal X\to B$ be the family obtained by blowing up
$\mathbb F_0\times B$ along ten pairwise disjoint labeled sections.
Use the ruling and exceptional classes to identify the real
Picard spaces of its fibers. Suppose $X_0$ contains a smooth
rational curve $\Gamma_0$ of class $\gamma$ with
$\gamma^2=-2k$ and normal bundle
$\OO_{\PP^1}(-k)^{\oplus2}$ in $\mathcal X$, where $k\in\ZZ_{>0}$.
If $D\in\Pic(X_0)\otimes_{\ZZ}\RR$ is nef, then the marked class
\[
s_\gamma(D)=D+\frac{D\cdot\gamma}{k}\,\gamma
\]
is nef on very general fibers.
\end{lemma}
\begin{proof}
Let $c$ be an effective integral class on the geometric generic
fiber $X_{\bar\eta}$. We will prove that $s_\gamma(D)\cdot c\geq 0$.
We first construct an effective divisor of class $c$ on $X_0$.
The two ruling classes and the ten exceptional classes are
represented by global line bundles on $\mathcal X$, and their
restrictions form a basis of the Picard group of each fiber.
Thus $c$ determines a global line bundle $\mathcal L$ whose
restriction to $X_{\bar\eta}$ is effective.
Let $K=\CC(B)$ and let $\overline K$ be its algebraic closure.
Flat base change along $K\subset\overline K$ gives
$H^0(X_\eta,\mathcal L_\eta)\otimes_K\overline K
\simeq H^0(X_{\bar\eta},\mathcal L_{\bar\eta})\neq 0$.
We may therefore choose a nonzero section of $\mathcal L_\eta$
over $K$ and denote its effective zero divisor by $C_\eta$.

Let $\mathcal C$ be the closure of $C_\eta$ in $\mathcal X$,
with its multiplicities. It is Cartier because $\mathcal X$
is smooth, and it contains no component of $X_0$.
Consequently, $C_0:=\mathcal C|_{X_0}$ is an effective divisor.
Near $0$, the line bundles $\mathcal O_{\mathcal X}(\mathcal C)$
and $\mathcal L$ differ only by a multiple of the fiber $X_0$,
whose restriction to $X_0$ is trivial. Hence
$\mathcal O_{X_0}(C_0)\simeq\mathcal L|_{X_0}$,
so $C_0$ represents the specialized class $c$.
Let $\mathcal I$ be the ideal sheaf of $\Gamma_0$ in $\mathcal X$,
put $N=N_{\Gamma_0/\mathcal X}$, and write
$\mathcal L=\OO_{\mathcal X}(\mathcal C)$.
The divisor $\mathcal C$ is defined by a section
$\sigma\in H^0(\mathcal X,\mathcal L)$.
Its order of vanishing $\mu$ along $\Gamma_0$ is the largest
integer such that
$\sigma\in H^0(\mathcal X,\mathcal L\otimes\mathcal I^\mu)$.
Taking its image modulo $\mathcal I^{\mu+1}$ gives a nonzero
section
\[
\overline\sigma\in
H^0\!\left(
\Gamma_0,\,
\mathcal L|_{\Gamma_0}\otimes
\mathcal I^\mu/\mathcal I^{\mu+1}
\right).
\]
Since $\Gamma_0$ is regularly embedded in the smooth threefold,
one has $\mathcal I/\mathcal I^2\simeq N^\vee$ and
$\mathcal I^\mu/\mathcal I^{\mu+1}\simeq\Sym^\mu N^\vee$.
Thus $\overline\sigma$ defines a nonzero homomorphism
$\mathcal L^{-1}|_{\Gamma_0}\to\Sym^\mu N^\vee$.
Concretely, choose local normal coordinates $v,t$ with
$X_0=\{t=0\}$ and $\Gamma_0=\{v=t=0\}$.
A local equation of $\mathcal C$ belongs to $(v,t)^\mu$,
and its image modulo $(v,t)^{\mu+1}$ is its first nonzero
homogeneous term in the normal variables $v,t$.
The section $\overline\sigma$ is the global expression of
these local terms.
Now $\deg(\mathcal L|_{\Gamma_0})=c\cdot\gamma$ and
$N\simeq\OO_{\PP^1}(-k)^{\oplus2}$, so the preceding
homomorphism becomes
\[
\OO_{\PP^1}(-c\cdot\gamma)
\longrightarrow
\Sym^\mu N^\vee
\simeq\OO_{\PP^1}(k\mu)^{\oplus(\mu+1)}.
\]
At least one component is nonzero. Since a nonzero map
$\OO_{\PP^1}(a)\to\OO_{\PP^1}(b)$ requires $a\leq b$,
we obtain $-c\cdot\gamma\leq k\mu$.
Moreover, restricting a local equation in $(v,t)^\mu$ to
$t=0$ gives an equation divisible by $v^\mu$.
Hence $C_0$ contains $\Gamma_0$ with multiplicity at least
$\mu$, and $C_0-\mu\Gamma_0$ is effective.
Together with $\mu+(c\cdot\gamma)/k\geq0$, this shows that
$s_\gamma(c)$ is represented by the effective $\mathbb Q$-divisor
\[
(C_0-\mu\Gamma_0)
+\left(\mu+\frac{c\cdot\gamma}{k}\right)\Gamma_0.
\]
Self-adjointness of $s_\gamma$ with respect to the intersection
pairing and nefness of $D$ give
\[
s_\gamma(D)\cdot c
=D\cdot s_\gamma(c)\geq0.
\]
This proves nefness on the geometric generic fiber and hence
on very general fibers.
\end{proof}
Keep $G$ fixed and move $(1,1)$ to $p_t=(1+t,1)$. The orbit
$G\cdot p_t$ consists of the ten points
\begin{align*}
p_\ell(t)
&=r^\ell\cdot p_t
=\bigl(\zeta^\ell(1+t),\zeta^{3\ell}\bigr),\\
p_{\ell+5}(t)
&=r^\ell s\cdot p_t
=\left(\frac{u\zeta^\ell}{1+t},u^3\zeta^{3\ell}\right),
\qquad 0\leq\ell\leq4.
\end{align*}
Choose an open neighborhood $B\subset\mathbb A^1$ of $0$ on which
these points are defined and pairwise distinct, and let $\mathcal X$
be the blow-up of ${\mathbb F_0}\times B$ along the corresponding ten sections.
Write $\Gamma_0$ for the strict transform of
$\Gamma\times\{0\}$ in $\mathcal X$, where $\Gamma\subset\mathbb F_0$ is the smooth rational
curve given in affine coordinates by $z=x^3$.
\begin{lemma}\label{lem:normal}
Let $\mathcal X\to B$ be the blow-up of $\mathbb F_0\times B$
along the ten labeled sections chosen above, and let
$\Gamma_0\subset X_0$ be the smooth rational curve constructed
above in the special fiber. Then
\[
N_{\Gamma_0/\mathcal X}\simeq\OO_{\PP^1}(-2)^{\oplus2}.
\]
\end{lemma}
\begin{proof}
We compute the changes in the normal bundle induced by blowing
up the labeled sections, following \cite[Section~3]{Petrakiev}.
Each section meeting the curve induces an elementary transform
of its normal bundle, with a one-dimensional quotient supported
at the intersection point and determined by the tangent
direction of the section. Since $\Gamma\simeq\PP^1$ has
self-intersection six in ${\mathbb F_0}$, its normal bundle is $\OO_{\PP^1}(6)$.
The parameter direction adds a trivial summand, so before blowing up
the normal bundle is $\OO_{\PP^1}(6)\oplus\OO_{\PP^1}$.
Let $P_1,\ldots,P_{10}$ be the ten points of $G\cdot(1,1)$,
which are the blow-up centers on the central fiber.
Write $P_i=(x_i,x_i^3)$, with the labeling
\[
x_{\ell+1}=\zeta^\ell,\qquad
x_{\ell+6}=u\zeta^\ell,\qquad 0\leq\ell\leq4.
\]
On the affine chart we use the normal coordinates $n=z-x^3,t$.
Near $P_i$, write the corresponding moving section as
$x=x_i+A(t)$, $n=B(t)$, where $A(0)=B(0)=0$, and put
$v_i=B'(0)$.
Its ideal is generated by
$\xi=x-x_i-A(t)$ and $m=n-B(t)$. In the blow-up chart $m=\xi q$,
the blow-down is
\[
(\xi,q,t)\longmapsto
\bigl(x_i+\xi+A(t),\,B(t)+\xi q,\,t\bigr).
\]
The strict transform of $\Gamma$ is $q=t=0$. Differentiating and
quotienting by the tangent direction to this curve gives the normal
map $(h,g)\mapsto(\xi h+v_i g,g)$.
Its image consists precisely of pairs $(f,g)$ with
$f-v_i g$ divisible by $\xi$, or equivalently
$f(x_i)=v_i g(x_i)$.
Thus, writing $N=N_{\Gamma_0/\mathcal X}$, we obtain
\[
0\longrightarrow N\longrightarrow
\OO_{\PP^1}(6)\oplus\OO_{\PP^1}
\longrightarrow\bigoplus_{i=1}^{10}\CC_{x_i}\longrightarrow0.
\]
At each point $x_i$, the quotient map imposes one linear
condition on the fiber of $\OO_{\PP^1}(6)\oplus\OO_{\PP^1}$.
Thus each elementary transform decreases the degree by one,
and the exact sequence gives $\deg N=6-10=-4$.
Here $x(t)$ and $z(t)$ denote the coordinates of a moving point,
and primes denote differentiation with respect to $t$.
For each of the first five sections,
$x(t)=\zeta^\ell(1+t)$, so
\[
x'(0)=\zeta^\ell=x(0).
\]
For each of the last five sections,
$x(t)=u\zeta^\ell/(1+t)$, and therefore
\[
x'(t)=-\frac{u\zeta^\ell}{(1+t)^2},
\qquad
x'(0)=-u\zeta^\ell=-x(0).
\]
In both cases the second coordinate is independent of $t$:
it is $\zeta^{3\ell}$ for the first five sections and
$u^3\zeta^{3\ell}$ for the last five. Hence $z'(0)=0$
along all ten sections.
Since $B(t)=z(t)-x(t)^3$, differentiation gives
$v_i=B'(0)=z'(0)-3x(0)^2x'(0)$.
Writing $x_i=x(0)$, we obtain $v_i=-3x_i^3$
when $x_i^5=1$, and $v_i=3x_i^3$ when $x_i^5=u^5$.
After twisting the preceding exact sequence by
$\OO_{\PP^1}(1)$, a global section of $N(1)$ is a section
of $\OO_{\PP^1}(7)\oplus\OO_{\PP^1}(1)$ whose image in
each of the ten point quotients vanishes.
On the affine chart, it is therefore represented by a single
pair of polynomials $f(x),g(x)$, with $\deg f\leq7$ and
$\deg g\leq1$, satisfying $f(x_i)=v_i g(x_i)$ for every $i$.
For the five points with $x_i^5=1$, these conditions read
$f(x_i)+3x_i^3g(x_i)=0$.
Thus the polynomial $f(x)+3x^3g(x)$ vanishes at all five
distinct roots of $x^5-1$. Consequently $x^5-1$ divides $f+3x^3g$.
Likewise, the conditions at the other five points say that
$f(x)-3x^3g(x)$ vanishes at all five distinct roots of
$x^5-u^5$, so it is divisible by $x^5-u^5$.
Hence the ten pointwise conditions are equivalent to the
two polynomial identities
\[
f+3x^3g=(x^5-1)a,\qquad
f-3x^3g=(x^5-u^5)b.
\]
Both polynomials on the left have degree at most seven.
The quotient polynomials therefore satisfy
$\deg a,\deg b\leq2$.
Subtracting, the coefficients in degrees five through seven force
$a=b$. Hence $6x^3g=(u^5-1)a$.
The left side has only degrees three and four, while the right side
has degree at most two. Thus $g=a=f=0$, proving $H^0(N(1))=0$.
By the splitting theorem,
$N\simeq\OO_{\PP^1}(r)\oplus\OO_{\PP^1}(s)$.
The vanishing gives $r,s\leq-2$, and their sum is $-4$;
therefore $r=s=-2$.
\end{proof}
\subsection{Conclusion of the proof}\label{sec:main-proofs}
We retain the marking and the family $\mathcal X\to B$ constructed
in Subsections~\ref{sec:nef} and~\ref{sec:reflection}.
\begin{proof}[Proof of Theorem~\ref{thm:general-points}]
On the central fiber $X_0=\Bl_{G\cdot(1,1)}\mathbb F_0$,
Proposition~\ref{prop:nef} gives the nef class
\[
R=\frac{5-\sqrt5}{2}H_1+\frac{5+\sqrt5}{2}H_2-E.
\]
The strict transform $\Gamma_0$ of $\Gamma=\{z=x^3\}$ has class
$\gamma=3H_1+H_2-E$, with $\gamma^2=-4$ and
$R\cdot\gamma=\sqrt5$. Lemmas~\ref{lem:reflection} and~\ref{lem:normal}
give the nefness of
\[
s_\gamma(R)=R+\frac{\sqrt5}{2}\gamma
=\left(\frac52+\sqrt5\right)
\left(H_1+H_2-\frac{E}{\sqrt5}\right)
\]
on a very general moving fiber. Nefness propagates from one such
fiber to the universal family of free $G$-orbits, and then to the
family of all ordered ten-tuples of distinct points of ${\mathbb F_0}$.
In each case the resulting lower bound is the volume bound
$\sqrt{2/10}=1/\sqrt5$.
\end{proof}
\section{Proof of Theorem~\ref{thm:plane-nine-points}}\label{sec:proof-theorem-two}
Write $S=\mathbb F_0$. Throughout this section we use the
coordinates $x'=x/u^{1/2}$ and $z'=u^{3/2}/z$, where $x,z$
on the right are the coordinates in~\eqref{eq:action}.
Dropping the primes, the generators of $G$ become
\[
r(x,z)=(\zeta x,\zeta^2z),\qquad s(x,z)=(1/x,1/z).
\]
This change of coordinates conjugates the original action and
preserves $\OO_S(1,1)$, so Theorem~\ref{thm:general-points}
applies to this action as well.

\subsection{The quotient surface}
Let $q\colon S\to Y=S/G$ be the quotient.
We first descend the polarization to $Y$, then identify its
pullback on a plane model of the minimal resolution and pass
from the resulting special configuration to very general points.

\begin{corollary}\label{thm:quotient}
The divisor $-5K_S$ descends to an ample Cartier divisor $A$ on
$Y$ with $A^2=20$ and $\eps(A;y)=2\sqrt5$ at a very general
smooth point $y\in Y$.
\end{corollary}
\begin{proof}
The quotient $Y$ is normal and projective.
Every nonidentity rotation fixes precisely the four corners
$(0,0)$, $(0,\infty)$, $(\infty,0)$, and $(\infty,\infty)$.
Each has stabilizer $\langle r\rangle$ of order five, since
every reflection interchanges $0$ and $\infty$ in both factors.
For $k=0,\ldots,4$, the reflection $r^ks$ acts by
$(x,z)\mapsto(\zeta^k/x,\zeta^{2k}/z)$, so its fixed points are
\[
P_{k,\epsilon,\delta}
=\bigl(\epsilon\zeta^{3k},\delta\zeta^k\bigr),
\qquad \epsilon,\delta\in\{1,-1\}.
\]
These twenty torus points are distinct and have stabilizer
$\operatorname{Stab}_G(P_{k,\epsilon,\delta})=\langle r^ks\rangle$.
Indeed, a point fixed by two distinct reflections would be
fixed by their product, a nonidentity rotation, which has no
fixed points in the torus. Together with the four corners,
these are all the points with nontrivial stabilizer.

The natural action on the fiber of $\OO_S(-K_S)$ at a fixed
point is the determinant of the tangent representation.
At a corner this character has order dividing five, whereas
at a reflection point the tangent action is
$\operatorname{diag}(-1,-1)$ and its determinant is one.
Thus every stabilizer acts trivially on the fiber of the
naturally linearized bundle
$\OO_S(-5K_S)\simeq\OO_S(10,10)$.
Finite quotient descent~\cite[Chapter~1, \S3]{MumfordFogartyKirwan}
gives a Cartier divisor $A$ on $Y$ with $q^*A\sim-5K_S$.
It is ample because its pullback is ample, and
$A^2=(q^*A)^2/\deg q=200/10=20$.

Let $p\in S$ have trivial stabilizer and put $y=q(p)$.
Then $y\in Y_{\mathrm{reg}}$ and $q^{-1}(y)=G\cdot p$.
Since $q$ is \'etale near this orbit, blowing up $y$ and its
inverse image gives a finite map
$\Bl_{G\cdot p}S\to\Bl_yY$ pulling the exceptional curve
over $y$ back to the sum of the ten exceptional curves upstairs.
Nefness is equivalent to nefness of its finite pullback; hence
\[
\eps(A;y)=\eps(\OO_S(10,10);G\cdot p)=2\sqrt5
\]
for a very general free orbit, by
Theorem~\ref{thm:general-points}.
The excluded proper closed subsets of $S$ have proper closed
images under the finite quotient, giving the very general
assertion on $Y$.
\end{proof}

\subsection{The plane model and the passage to very general points}
\begin{proof}[Proof of Theorem~\ref{thm:plane-nine-points}]
Let $\nu\colon Z\to Y$ be the minimal resolution and put
$M=\nu^*A$.
The surface $Y$ has exactly six singular points.
The four corners form two orbits,
$\{(0,0),(\infty,\infty)\}$ and
$\{(0,\infty),(\infty,0)\}$.
Their tangent weights are $(1,2)$ and $(1,3)$ modulo five,
respectively; both quotient singularities have type
$\frac15(1,2)$, since $2$ and $3$ are inverse modulo five.
The twenty reflection points form four orbits of size five,
represented by $(\epsilon,\delta)$ with
$\epsilon,\delta\in\{1,-1\}$.
Their images are four $A_1$ singularities, since the stabilizer
acts on the tangent space as $\operatorname{diag}(-1,-1)$.
The action is free elsewhere, so there are no other singularities.

The invariant function $x^5+x^{-5}$ defines a morphism
$Y\to\PP^1$. Composing with $\nu$ gives
$f\colon Z\to\PP^1$, whose general fiber is $\PP^1$.
Its reducible fibers lie over $2,-2,\infty$.
Over each of $2$ and $-2$, the reduced fiber on $Y$ contains
two $A_1$ points; over infinity it contains the two cyclic
quotient points. Index the finite reducible fibers by $j=1,2$,
corresponding to $2$ and $-2$, respectively.
The fibers on $Z$ have the following chains, with components
listed in their order along each chain:
\[
\begin{array}{c|c|c}
\text{components}&\text{self-intersections}&\text{multiplicities}\\
\hline
U_j,V_j,W_j\ (j=1,2)&(-2,-1,-2)&(1,2,1)\\
G_1,G_2,T,G_3,G_4&(-3,-2,-1,-3,-2)&(1,3,5,2,1).
\end{array}
\]
Here $U_j,W_j,G_1,G_2,G_3,G_4$ are the eight exceptional
curves of $\nu$, whereas $V_1,V_2,T$ are the strict transforms
of the reduced fibers on $Y$.
The cyclic chains follow from the negative continued fraction
$5/2=[3,2]$, with opposite orientations relative to $T$ at
the two cyclic points. The components $V_j$ and $T$ have
multiplicities two and five, respectively, as seen from the
ramification indices of the map $x\mapsto x^5+x^{-5}$.
The other multiplicities and the self-intersections
$V_j^2=T^2=-1$ follow by intersecting the whole fiber with
each component.

Let $B=\{z=x^2\}\subset S$, let $\overline B=q(B)$, and
let $C\subset Z$ be its strict transform.
The curve $B$ is $G$-invariant, and its quotient map is
$x\mapsto x^5+x^{-5}$; hence $C$ is a section of $f$.
Since $q^*\overline B=B$, one has
$\overline B^{\,2}=B^2/10=2/5$.
The curve $\overline B$ passes through one $A_1$ point in each
finite reducible fiber and through the cyclic quotient point
corresponding to the orbit $\{(0,0),(\infty,\infty)\}$.
In each finite reducible fiber, label the two exceptional
components so that $C$ meets $U_j$ and is disjoint from $W_j$.
At the cyclic point, $C$ meets the $(-3)$-component $G_1$ of
the chain $G_1,G_2$. Indeed, in the resolution chart with
coordinates $v=z/x^2$, $w=x^5$, the curve $G_1$ is given by
$w=0$ and $C$ by $v=1$. Thus $C$ meets $G_1$ transversely,
away from the other fiber components. The intersections with
$U_1$ and $U_2$ are likewise transverse and occur at smooth
points of the reduced fibers.
The corrections to the self-intersection are $1/2$ at each
$A_1$ point and $2/5$ at the cyclic point. The latter is the
diagonal entry corresponding to the $(-3)$-curve in the
inverse of the negative intersection matrix of the chain
$(-3,-2)$. Therefore
\[
C^2=\frac25-\frac12-\frac12-\frac25=-1.
\]
The following schematic picture shows the three reducible fibers
and the section $C$ on $Z$, together with the map to the base $\mathbb P^1$.
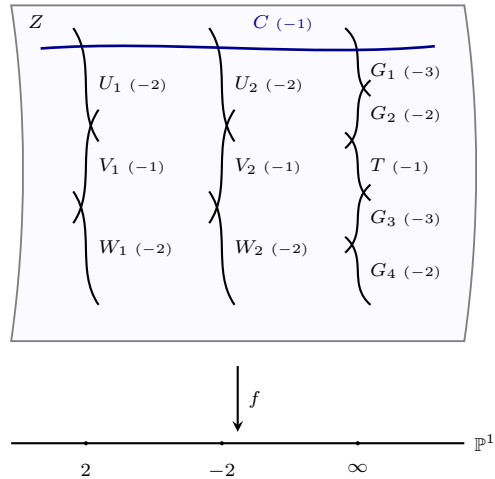
\begin{center}
\begin{tikzpicture}[scale=.6,
  x=1cm,y=1cm,
  component/.style={line width=.8pt},
  section/.style={draw=blue!55!black,line width=1pt},
  every node/.style={font=\tiny}
]
\path[draw=black!50,fill=blue!2,line width=.7pt]
  (0,.6)
  -- (10,.6)
  .. controls (10.35,3) and (10.35,5.8) .. (10,8)
  -- (0,8)
  .. controls (.35,5.8) and (.35,3) .. cycle;
  \node at (.55,7.65) {$Z$};

  \foreach \x/\j in {1.65/1,4.65/2} {
    \foreach \top/\s/\name/\sq in {
      7.5/1/U/-2,5.7/-1/V/-1,3.9/1/W/-2
    } {
      \draw[component]
        ({\x-.28*\s},\top)
        .. controls ({\x+.24*\s},{\top-.75})
                and ({\x-.24*\s},{\top-1.75})
        .. ({\x+.28*\s},{\top-2.5});
      \node[right] at ({\x+.05},{\top-1.25})
        {$\name_{\j}\;{\scriptstyle(\sq)}$};
    }
  }

  \foreach \top/\s/\name/\sq/\labely in {
    7.5/1/{G_1}/-3/6.55,
    6.35/-1/{G_2}/-2/5.6,
    5.2/1/T/-1/4.45,
    4.05/-1/{G_3}/-3/3.3,
    2.9/1/{G_4}/-2/2.15
  } {
    \draw[component]
      ({7.65-.28*\s},\top)
      .. controls ({7.65+.24*\s},{\top-.45})
              and ({7.65-.24*\s},{\top-1.05})
      .. ({7.65+.28*\s},{\top-1.5});
    \node[right] at (7.7,\labely)
      {$\name\;{\scriptstyle(\sq)}$};
  }

  \draw[section]
    (.65,7.05)
    .. controls (3.3,7.25) and (6.7,6.85) .. (9.35,7.1);
  \node[above,text=blue!55!black] at (6,7.2)
    {$C\;{\scriptstyle(-1)}$};

  \draw[->,>=stealth,line width=.8pt]
    (5,.05)--(5,-1.4) node[midway,right] {$f$};
  \draw[line width=.8pt] (0,-1.65)--(10,-1.65);
  \node[right] at (10,-1.65) {$\mathbb P^1$};
  \foreach \x/\value in {1.65/2,4.65/-2,7.65/\infty} {
    \fill (\x,-1.65) circle (1.3pt);
    \node[below=4pt] at (\x,-1.65) {$\value$};
  }
\end{tikzpicture}
\end{center}
In each finite reducible fiber, contract $V_j$ and then the
image of $W_j$. In the fiber over infinity, contract
successively $T,G_2,G_3,G_4$.
Each curve has self-intersection $-1$ when it is contracted,
and these eight contractions are disjoint from the successive
images of $C$. The resulting ruled surface has a section of
square $-1$, so it is $\mathbb F_1$.
Contracting this section gives a birational morphism
$Z\to\PP^2$ whose inverse is a sequence of nine point blow-ups.
Let $h,e_1,\ldots,e_9$ be the resulting plane marking, with
$h$ the pullback of a line and each $e_i$ a total exceptional
class. The first blow-up corresponds to the section, so
$C=e_1$. Write $F$ for the class of a general fiber of $f$.
Labeling the remaining blow-ups in the reverse order of the
contractions within each fiber gives the classes of all fiber
components:
\begin{align*}
F&=h-e_1,\\
U_j&=h-e_1-e_{2j}-e_{2j+1},
&V_j&=e_{2j+1},
&W_j&=e_{2j}-e_{2j+1}\qquad(j=1,2),\\
G_1&=h-e_1-e_6-e_7-e_8,
&G_2&=e_8-e_9,
&T&=e_9,\\
G_3&=e_7-e_8-e_9,
&G_4&=e_6-e_7.
\end{align*}
In particular, the weighted sums of the components recover
the fiber class:
\[
U_j+2V_j+W_j
=G_1+3G_2+5T+2G_3+G_4=F.
\]
The curves $V_j$ and $T$ appear in this marking because they
are contracted in the passage to the ruled surface; they are
not contracted by $\nu$, so the orthogonality to $M=\nu^*A$
applies only to $U_1,W_1,U_2,W_2,G_1,G_2,G_3,G_4$.
The projection formula gives $M\cdot C=(q^*A\cdot B)/10=3$.
Moreover, $M\cdot F=10$, since the pullback to $S$ of a
general fiber on $Y$ is the sum of ten first-ruling fibers.
Together with the orthogonality of $M$ to the eight exceptional
curves of $\nu$, these equations determine its class:
\begin{equation}\label{eq:plane-degree-thirteen}
M=13h-3e_1-5(e_2+e_3+e_4+e_5)
  -4(e_6+e_7)-2(e_8+e_9).
\end{equation}
Choose $y\in Y_{\mathrm{reg}}$ with $\eps(A;y)=2\sqrt5$
and put $a=\nu^{-1}(y)$.
Let $\pi\colon\widehat Z\to Z$ be the blow-up of $a$, with
exceptional class $e_{10}$.
The class $\pi^*M-2\sqrt5\,e_{10}$ is nef, since it is the
pullback of the corresponding nef class on $\Bl_yY$.
Thus, using total pullbacks for the plane marking, the class
\[
D=13h-3e_1-5(e_2+e_3+e_4+e_5)
  -4(e_6+e_7)-2(e_8+e_9)-2\sqrt5\,e_{10}
\]
is nef on $\widehat Z$.
Consider the parameter space $\mathcal B_{10}$ of ordered
sequences of ten point blow-ups of $\PP^2$, allowing infinitely
near centers. At each stage the next center is chosen on the
universal blown-up surface. This constructs a smooth
irreducible parameter space and a smooth projective family
whose fibers carry the line class and the ten total exceptional
classes. The locus of ten distinct points of $\PP^2$ is a
dense open subset, and $\widehat Z$ is a fiber of this family.
The marked class $D$ is nef on very general fibers.
Indeed, every integral marked class defines a line bundle on
the universal family. For any such class $c$ with $D\cdot c<0$,
the locus where $c$ is effective is closed by upper
semicontinuity of $h^0$. It does not contain the fiber
$\widehat Z$, since $D$ is nef there, so it is a proper closed
subset of $\mathcal B_{10}$.
There are only countably many integral marked classes.
Outside the union of these proper closed subsets, no curve
has negative intersection with $D$, proving the assertion.
In particular, $D$ is nef on the blow-up of ten very general
distinct plane points.
Apply two quadratic transformations, each centered at three
of the first nine points with largest multiplicities, and
reorder those nine points after each step. The first nine
coefficients change as follows:
\[
(13;5^4,4^2,3,2^2)
\longmapsto(11;5,4^2,3^4,2^2)
\longmapsto(9;3^5,2^4),
\]
where superscripts denote repetitions.
These transformations lift to isomorphisms between the
corresponding blown-up surfaces and induce birational maps
of the configuration spaces. The tenth point is not a center
of either transformation, and its coefficient remains
$2\sqrt5$. Consequently the class
\[
9H-3(E_1+\cdots+E_5)-2(E_6+\cdots+E_9)
  -2\sqrt5\,E_{10}
\]
is nef on the blow-up of ten very general plane points.
Since the map forgetting the tenth point is smooth and open,
the countably many proper closed exceptional loci above
restrict to proper closed subsets of its fiber over a very
general nine-point configuration.
Fix such a configuration, set
$X=\Bl_{p_1,\ldots,p_9}\PP^2$, and let $L$ be the divisor
in the statement.
Then, for a very general point $x\in X$, the class
$\pi_x^*L-2\sqrt5\,E_x$ is nef, where
$\pi_x\colon\Bl_xX\to X$ is the blow-up at $x$ and $E_x$
is its exceptional curve.

It remains to prove that $L$ is ample.
The unique cubic through nine very general plane points is
smooth and irreducible. Its strict transform represents
$-K_X$ and has square zero; it intersects every distinct
irreducible curve nonnegatively, so $-K_X$ is nef.
Let $\eta\colon X\to S_5=\Bl_{p_1,\ldots,p_5}\PP^2$
contract $E_6,\ldots,E_9$.
The surface $S_5$ is a del Pezzo surface of degree four, and
\[
L=2(-K_X)+\eta^*(-K_{S_5}).
\]
Thus $L$ has positive intersection with every curve not
contracted by $\eta$, while $L\cdot E_i=2$ for $6\leq i\leq9$.
Since $L^2=81-5\cdot9-4\cdot4=20$, the Nakai--Moishezon
criterion proves ampleness.
The nef class constructed above gives
$\eps(L;x)\geq2\sqrt5$ at a very general point, and the
volume bound $\eps(L;x)\leq\sqrt{L^2}=2\sqrt5$ gives equality.
\end{proof}


\begin{thebibliography}{99}
\bibitem{DionneRoth}
\textsc{C. Dionne and M. Roth},
Seshadri constants on $\mathbb P^1\times\mathbb P^1$ and applications
to the symplectic packing problem,
\emph{Forum Math. Sigma} \textbf{13} (2025), e191, 42 pp.
\href{https://doi.org/10.1017/fms.2025.10137}
{\nolinkurl{doi:10.1017/fms.2025.10137}}.
\bibitem{DumnickiKuronyaMacleanSzemberg}
\textsc{M. Dumnicki, A. K\"uronya, C. Maclean, and T. Szemberg},
Rationality of Seshadri constants and the
Segre--Harbourne--Gimigliano--Hirschowitz conjecture,
\emph{Adv. Math.} \textbf{303} (2016), 1162--1170.
\href{https://doi.org/10.1016/j.aim.2016.05.025}
{\nolinkurl{doi:10.1016/j.aim.2016.05.025}}.
\bibitem{HanumanthuHarbourne}
\textsc{K. Hanumanthu and B. Harbourne},
Single point Seshadri constants on rational surfaces,
\emph{J. Algebra} \textbf{499} (2018), 37--42.
\href{https://doi.org/10.1016/j.jalgebra.2017.11.040}
{\nolinkurl{doi:10.1016/j.jalgebra.2017.11.040}}.
\bibitem{HanumanthuJacobSuhasSingh}
\textsc{K. Hanumanthu, C. J. Jacob, Suhas B. N., and A. K. Singh},
Rationality of Seshadri constants on blow-ups of ruled surfaces,
\emph{Manuscripta Math.} \textbf{177} (2026), article 32.
\href{https://doi.org/10.1007/s00229-026-01713-7}
{\nolinkurl{doi:10.1007/s00229-026-01713-7}}.
\bibitem{Lazarsfeld}
\textsc{R. Lazarsfeld},
\emph{Positivity in Algebraic Geometry I: Classical Setting:
Line Bundles and Linear Series},
Ergeb. Math. Grenzgeb. (3) \textbf{48}, Springer-Verlag, Berlin, 2004.
\href{https://doi.org/10.1007/978-3-642-18808-4}
{\nolinkurl{doi:10.1007/978-3-642-18808-4}}.
\bibitem{MumfordFogartyKirwan}
\textsc{D. Mumford, J. Fogarty, and F. Kirwan},
\emph{Geometric Invariant Theory}, third ed.,
Ergeb. Math. Grenzgeb. (2) \textbf{34}, Springer-Verlag, Berlin, 1994.
\href{https://link.springer.com/book/9783540569633}
{Springer book page}.
\bibitem{Petrakiev}
\textsc{I. Petrakiev},
Homogeneous interpolation and some continued fractions,
\emph{Trans. Amer. Math. Soc. Ser. B} \textbf{1} (2014), 23--44.
\href{https://doi.org/10.1090/S2330-0000-2014-00001-0}
{\nolinkurl{doi:10.1090/S2330-0000-2014-00001-0}}.
\end{thebibliography}
\end{document}